\documentclass{birkau}

\usepackage{amsmath,amssymb,url,tikz,mathrsfs,hyperref}
\usepackage[mathcal]{euscript} 

\usetikzlibrary{arrows,shapes,positioning}
\usetikzlibrary{decorations.markings}
\tikzstyle arrowstyle=[scale=1]
\tikzstyle directed=[postaction={decorate,decoration={markings,
    mark=at position 0.535 with {\arrow[arrowstyle]{to}}}}]

\numberwithin{equation}{section}

\theoremstyle{plain}
\newtheorem{theorem}{Theorem}[section]
\newtheorem{lemma}[theorem]{Lemma}
\newtheorem{proposition}[theorem]{Proposition}
\newtheorem{corollary}[theorem]{Corollary}

\theoremstyle{definition}
\newtheorem{definition}[theorem]{Definition}
\newtheorem{notation}[theorem]{Notation}

\newtheorem{problem}{Problem}

\renewcommand{\c}[1]{\mathcal{#1}}
\renewcommand{\bf}[1]{\mathbf{#1}}
\newcommand{\bb}[1]{\mathbb{#1}}
\newcommand{\trm}[1]{\textrm{#1}}
\newcommand{\s}[1]{\mathscr{#1}}
\newcommand{\inn}[1]{\langle #1\rangle}
\newcommand{\deq}{\mathrel{\mathop:}=}
\newcommand{\nmodels}{\not\models}
\renewcommand{\le}{\leqslant}
\renewcommand{\ge}{\geqslant}
\renewcommand{\b}{\breve{\ }}
\newcommand{\m}{\wedge}
\renewcommand{\j}{\vee}
\newcommand{\pw}{\raise0.515ex\hbox{{\scalebox{1.125}[1.125]{\(\wp\)}}}}
\DeclareMathOperator{\Var}{Var}
\DeclareMathOperator{\Si}{Si}
\DeclareMathOperator{\ISU}{\bb{ISU}}
\DeclareMathOperator{\Cm}{\bf{Cm}}
\DeclareMathOperator{\I}{\bb{I}}
\DeclareMathOperator{\IS}{\bb{IS}}
\DeclareMathOperator{\So}{\bb{S}}
\DeclareMathOperator{\U}{\bb{U}}

\begin{document}

\title[Varieties of semiassociative relation algebras and tense algebras]{Varieties of semiassociative relation algebras and tense algebras}

\author[J. M. Koussas]{James M. Koussas}
\address{Department of Mathematics \& Statistics\\
La Trobe University\\Victoria 3086\\Australia}
\email{j.koussas@latrobe.edu.au}

\author[T. Kowalski]{Tomasz Kowalski}
\address{Department of Mathematics \& Statistics\\
La Trobe University\\Victoria 3086\\Australia}
\urladdr{https://tomasz-kowalski.github.io}
\email{t.kowalski@latrobe.edu.au}



\begin{abstract}
It is well known that the subvariety lattice of the variety of relation algebras has exactly three atoms. The (join-irreducible) covers of two of these atoms are known, but a complete classification of the (join-irreducible) covers of the remaining atom has not yet been found. These statements are also true of a related subvariety lattice, namely the subvariety lattice of the variety of semiassociative relation algebras. The present article shows that this atom has continuum many covers in this subvariety lattice (and in some related subvariety lattices) using a previously established term equivalence between a variety of tense algebras and a variety of semiassociative \(r\)-algebras.
\end{abstract}

\maketitle

\section{Introduction}\label{sec:intro}

Varieties have been a major focus of research in relation algebras for a number of years. For example, several of the early landmark results in the subject focus on the variety of representable relation algebras. In 1955, Tarski showed in  \cite{tar3} that the class of all representable relation algebras is indeed a variety. A year later, in \cite{lyndon}, Lyndon gave the first (equational) basis for this variety. In 1964, Monk showed in \cite{monk} that this variety is not finitely based. 

The lattice of subvarieties of the variety of all relation algebras was first studied extensively by J\'onsson in \cite{varRA}, although Tarski did publish some results much earlier in \cite{tarring}. In \cite{tarring}, Tarski showed, using a result from J\'onsson and Tarski \cite{bao2}, that this lattice has exactly three atoms. These atoms are generated by \(\bf{A}_1\), \(\bf{A}_2\), and \(\bf{A}_3\) (the minimal subalgebras of the full relation algebras on a 1-element set, a 2-element set, and a 3-element set, respectively). As the variety of relation algebras is congruence distributive, we only need to look for join-irreducible varieties to find the other varieties of small height. It follows from results in \cite{bao2} that the variety generated by \(\bf{A}_1\) has no join-irreducible covers. In \cite{freedisc}, Andr\'eka, J\'onsson, and N\'emeti showed that the variety generated by \(\bf{A}_2\) has exactly one join-irreducible cover. The variety generated by \(\bf{A}_3\) is known to have at least 20 finitely generated join-irreducible covers and at least one infinitely generated join-irreducible cover; generators of these covers are listed by Jipsen in \cite{jipphd}. However, it is not yet known if there are more covers. The problem of completely classifying these covers is posed (in various forms) by J\'onsson and Maddux in \cite{perspecRA}, by Hirsch and Hodkinson in \cite{rabg}, and by Givant in \cite{advRA}. Further results on this lattice can be found in J\'onsson \cite{varRA} and Andr\'eka, Givant, and N\'emeti \cite{dpfetoRA}, for example.

Semiassociative relation algebras arise fairly naturally in the study of relation algebras and the calculus of relations; see Maddux \cite{seqcalc}, for example. As such, the subvariety lattice of the variety of all semiassociative relation algebras has attracted interest from researchers in this field. In \cite{freedisc}, Andr\'eka, J\'onsson, and N\'emeti show that the properties of the subvariety lattice of the variety of relation algebras that we mentioned above also hold in this lattice. In \cite{tot}, Jipsen, Kramer, and Maddux show that \(\bf{A}_3\) has a countably infinite family of covers (in the semiassociative case) by constructing a term equivalence and a countably infinite family of covers of the variety generated by \(\bf{T}_0\) (the two element Boolean algebra with a pair of identity operators).  In 1999, the problem of showing that \(\bf{T}_0\) has continuum many covers was posed by Peter Jipsen in a conversation with the second author relating to Kowalski \cite{vartense}. In Section \ref{sec:vars} of the present article we will solve this problem, and hence conclude that \(\bf{A}_3\) has continuum many covers in the subvariety lattice of the variety of semiassociative relation algebras.

\section{Preliminaries}\label{sec:prelim}

\subsection{Relation-type algebras}\label{subsec:NRA}

For an extensive introduction to the theory and history of relation algebras, we refer the reader to Hirsch and Hodkinson \cite{rabg}, Maddux \cite{madduxRA} and Givant \cite{intRA}; shorter introductions can be found in Chin and Tarski \cite{chintar} and J\'onsson \cite{varRA}. For sources that cover nonassociative and semiassociative relation algebras, we refer the reader to Maddux \cite{varcontRA}, Maddux \cite{seqcalc}, Hirsch and Hodkinson~\cite{rabg}, and Maddux \cite{madduxRA}. Subvariety lattices are discussed in J\'onsson \cite{varRA}, Andr\'eka, J\'onsson, and N\'emeti \cite{freedisc}, Jipsen \cite{jipphd}, Andr\'eka, Givant, and N\'emeti \cite{subRRA}, Jipsen, Kramer, and Maddux \cite{tot}, Andr\'eka, Givant, and N\'emeti \cite{dpfetoRA}, and Givant \cite{advRA}. We will begin by recalling some definitions from \cite{varcontRA} and \cite{tot}.

\begin{definition}
An algebra \(\bf{A} = \inn{A;\j,\m,\cdot,{^\prime},\b,0,1,e}\) is called a \emph{nonassociative relation algebra} iff \(\inn{A;\j,\m,{^\prime},0,1}\) is a Boolean algebra, \(\cdot\) is a binary operation, \(e\) is an identity element for \(\cdot\), and the \emph{triangle laws} hold in \(\bf{A}\), i.e.,
\[
(x \cdot y) \m z = 0 \iff (x\b \cdot z) \m y = 0 \iff (z \cdot y\b) \m x = 0,
\]
for all \(x,y,z \in A\). A nonassociative relation algebra \(\bf{A}\) is called a \emph{semiassociative relation algebra} iff \(\bf{A}\) satisfies \((x \cdot 1) \cdot 1 \approx x \cdot 1\). A  nonassociative relation algebra is called \emph{reflexive} iff \(x \le x\cdot x\), for all \(x \in A\). A nonassociative relation algebra \(\bf{A}\) is called \emph{symmetric} iff \(\bf{A}\) satisfies \(x\b \approx x\). A nonassociative relation algebra \(\bf{A}\) is called \emph{subadditive} iff \(x \cdot (x' \m y) \le x \j y\), for all \(x,y \in A\).
\end{definition}

The triangle laws are equivalent to an equation modulo the other axioms of nonassociative relation algebras (see Theorem 1.4 of \cite{varcontRA}), so the classes of all nonassociative relation algebras,  semissociative relation algebras, and reflexive subadditive symmetric semiassociative relation algebras are varieties.

\begin{notation}
The subvariety lattices of the varieties of nonassociative relation algebras,
semiassociative relation algebras and reflexive subadditive symmetric
semiassociative algebras will be denoted by
\(\boldsymbol{\Lambda}_\textup{NA}\), \(\boldsymbol{\Lambda}_\textup{SA}\) and
\(\boldsymbol{\Lambda}_\textup{RSA}\), respectively.
\end{notation}

\subsection{Tense algebras}\label{subsec:tense}

For an introduction to the theory and history of tense (or temporal) algebras, we refer the reader to Kracht \cite{kracht} and Blackburn, de Rijke, and Venema \cite{modal}. We refer the reader to Hirsch and Hodkinson  \cite{rabg} and Maddux \cite{madduxRA} for an introduction to the theory of Boolean algebras with conjugated operators; J\'onsson and Tarski \cite{bao1} is a shorter introduction. We will need the following definitions from \cite{tot}. 

\begin{definition}
An algebra \(\bf{A} = \inn{A; \j, \m,{^\prime},f,g,0,1}\) is called a \emph{tense algebra} iff \(\inn{A; \j, \m, {^\prime}, 0, 1}\) is a Boolean algebra, and \(f\) and \(g\) are a conjugate pair, i.e.,
\[
f(x) \m y = 0 \iff x \m g(y) = 0,
\]
for all \(x,y \in A\). A tense algebra \(\bf{A}\) is called \emph{total} iff \(f(x) \j g(x) = 1\), for all \(x \in A\) with \(x \neq 0\).
\end{definition}

A concrete example of a tense algebra is the complex algebra of a frame, i.e., a directed graph. Recall that the \emph{complex algebra} of a frame \(\inn{U;R}\) is the algebra \(\Cm(\inn{U;R}) \deq \inn{ \pw(U); \cup, \cap, {^c}, f_R, g_R, \varnothing, U}\), where \(\pw(U)\), \({^c}\), \(f_R\), and \(g_R\) respectively denote the powerset of \(U\), complementation relative to \(U\), the image operation defined by \(R\), and the preimage operation defined by \(R\). Thus, we have
\begin{align*}
f_R(X) &= \{u \in U \mid (x,u) \in R, \trm{ for some } x \in X \},\\
g_R(X) &= \{u \in U \mid (u,x) \in R, \trm{ for some } x \in X \},
\end{align*}
for all \(X \subseteq U\).

The fact that \(f\) and \(g\) form a conjugate pair can be expressed by equations (see Theorem 1.15 of \cite{bao1}), so the class of all tense algebras is a variety. The class of all total tense algebras is not a variety, but the variety it generates is finitely based (see Jipsen \cite{jipdisc}).

\begin{notation}
The subvariety lattice of the variety of tense algebras will be denoted by \(\boldsymbol{\Lambda}_\trm{TA}\). The subvariety lattice of the variety generated by the class of all total tense algebras will be denoted by \(\boldsymbol{\Lambda}_\trm{TTA}\).
\end{notation}

The variety of reflexive subadditive symmetric semiassociative \(r\)-algebras (reflexive subadditive symmetric semiassociative relation algebras without the requirement of an identity element) is known to be term equivalent to \(\boldsymbol{\Lambda}_\trm{TTA}\) (see Theorem 7 of \cite{tot}). This observation can be used to establish the following result (see Section 4 of \cite{tot}). Here we use \(\Var(\bf{A})\) to denote variety generated by an algebra \(\bf{A}\).

\begin{proposition}\label{covprop}
\(\Var(\bf{A}_3)\) has at least as many covers \textup{(}in \(\boldsymbol{\Lambda}_\textup{RSA}\)\textup{)} as \(\Var(\bf{T}_0)\) \textup{(}in \(\boldsymbol{\Lambda}_\textup{TTA}\)\textup{)}.
\end{proposition}

We conclude this section with some standard results that we use later. We call a binary relation \(R\) on a set \(U\) \emph{total} when \((x,y) \in R\) or \((y,x) \in R\), for all \(x,y \in U\); we do not require that \(x \neq y\), so total relations are reflexive. The following results can be found in Section 1 of \cite{tot} and Chapter 2 of \cite{jipphd}.

\begin{proposition}\label{tense}
\begin{enumerate}
\item
Let \(R\) be a total binary relation on a set \(U\). Then \(\Cm(\inn{U;R})\) is a total tense algebra.
\item
Let \(\bf{A}\) be a total tense algebra. Then \(\bf{A}\) is discriminator.
\end{enumerate}
\end{proposition}

Lastly, we state some basic results on discriminator varieties; we refer the reader to Lemma 9.2 and Theorem 9.4 of \cite{BSUA}. Here we use \(\I\), \(\So\), and \(\U\) for the usual class operators of taking all isomorphic copies, subalgebras, and ultraproducts of a class of similar algebras, respectively.

\begin{proposition}\label{disc}
Let \(\c{K}\) be a class of similar non-trivial algebras with a common discriminator term. Then
\begin{enumerate}
\item
every element of \(\c{K}\) is simple,
\item
all directly indecomposable and subdirectly irreducible elements of \(\Var(\c{K})\) are simple,
\item
the class of simple (and therefore the class of subdirectly irreducible) elements of \(\Var(\c{K})\) is precisely \(\ISU(\c{K})\).
\end{enumerate}
\end{proposition}

\section{Uncountable collections of varieties}\label{sec:vars}

In this section we will construct continuum many varieties that cover \(\Var(\bf{T}_0)\) in \(\boldsymbol{\Lambda}_{\trm{TTA}}\) and hence show that \(\Var(\bf{A}_3)\) has continuum many covers in \(\boldsymbol{\Lambda}_{\trm{RSA}}\). These varieties are generated by the complex algebras of an uncountable collection of frames with a very strong resemblance to the recession frames that were defined by Blok in \cite{blok} and by Jipsen, Kramer, and Maddux in \cite{tot}. Roughly speaking, in the terminology of \cite{tot}, these frames are \(\omega\)-recession frames with minor modifications that are combined like a \(\bb{Z}\)-recession frame. Before we make this more explicit, we will need to introduce some notation.

\begin{notation}
Let \(\bb{E}\) and \(\bb{O}\) denote \(\{2n \mid n \in \bb{N}\}\) and \(\{2n+1 \mid n \in \bb{N}\}\), respectively (where \(\bb{N} \deq \{p \in \mathbb{Z} \mid p > 0\}\)). For each \(S \subseteq \bb{O}\), let \(S_\bb{E} \deq S \cup \bb{E}\).
\end{notation}

Now we can define the frames we will be working with. We use \(a_{p,m}\) for an arbitrary vertex (or node) because vertices correspond to atoms of complex algebras, and to match the notation used in \cite{tot}.

\begin{definition}\label{defn}
Let \(a_{p,m}\) be a (distinct) vertex, for each \(p \in \mathbb{Z}\) and \(m \in \mathbb{N}\). Define \(V \deq \{a_{p,m} \mid p \in \bb{Z}, m \in \bb{N}\}\). For each \(S \subseteq \bb{O}\), let
\begin{align*}
R_S \deq  \phantom{} & \{(a_{p,m}, a_{q,n}) \mid  p > q \trm{ or } (p=q \trm{ and } m \ge n) \} \\
 & \cup \{(a_{p,1}, a_{p+1,m}) \mid p \in \bb{Z},  m \in S_\bb{E} \}\\
 & \cup \{(a_{p,m}, a_{p,m+1}) \mid p \in \bb{Z}, m \in \bb{N}\}
\end{align*}
and let \(\bf{F}_S \deq \Cm(\inn{V;R_S})\).
\end{definition}

As is usually the case with frames, a diagram is easier to work with than a written description. Thus, we will usually refer to Figure \ref{fig}. 

\begin{figure}[h]
\begin{tikzpicture}[scale=4/3]
\draw[thick, directed] (4,0) -- (1,1.5);
\draw[thick, directed, dashed] (4,0) -- (2,1.5);
\draw[thick, directed, dashed] (4,0) -- (0,1.5);
\draw[thick, directed] (4,0) -- (3,1.5);
\draw[thick, directed] (4,1.5) -- (1,3);
\draw[thick, directed, dashed] (4,1.5) -- (2,3);
\draw[thick, directed, dashed] (4,1.5) -- (0,3);
\draw[thick, directed] (4,1.5) -- (3,3);
\draw[thick, directed] (4,-1.5) -- (1,0);
\draw[thick, directed, dashed] (4,-1.5) -- (2,0);
\draw[thick, directed, dashed] (4,-1.5) -- (0,0);
\draw[thick, directed] (4,-1.5) -- (3,0);
\draw[thick, directed] (4,-3) -- (1,-1.5);
\draw[thick, directed, dashed] (4,-3) -- (2,-1.5);
\draw[thick, directed, dashed] (4,-3) -- (0,-1.5);
\draw[thick, directed] (4,-3) -- (3,-1.5);
\draw[thick, directed] (1,0) -- (0,0);
\draw[thick, directed] (2,0) -- (1,0);
\draw[thick, directed] (3,0) -- (2,0);
\draw[thick, directed] (4,0) -- (3,0);
\draw[thick, directed] (1,1.5) -- (0,1.5);
\draw[thick, directed] (2,1.5) -- (1,1.5);
\draw[thick, directed] (3,1.5) -- (2,1.5);
\draw[thick, directed] (4,1.5) -- (3,1.5);
\draw[thick, directed] (2,3) -- (1,3);
\draw[thick, directed] (3,3) -- (2,3);
\draw[thick, directed] (4,3) -- (3,3);
\draw[thick, directed] (1,3) -- (0,3);
\draw[thick, directed] (1,-1.5) -- (0,-1.5);
\draw[thick, directed] (2,-1.5) -- (1,-1.5);
\draw[thick, directed] (3,-1.5) -- (2,-1.5);
\draw[thick, directed] (4,-1.5) -- (3,-1.5);
\draw[thick, directed] (1,-3) -- (0,-3);
\draw[thick, directed] (2,-3) -- (1,-3);
\draw[thick, directed] (3,-3) -- (2,-3);
\draw[thick, directed] (4,-3) -- (3,-3);
\draw (-0.5,0) node {\(\cdots\)};
\draw (-0.5,1.5) node {\(\cdots\)};
\draw (-0.5,3) node {\(\cdots\)};
\draw (-0.5,-1.5) node {\(\cdots\)};
\draw (-0.5,-3) node {\(\cdots\)};
\draw[thick, fill=white] (1,0) circle (0.05);
\draw[thick, fill=white] (2,0) circle (0.05);
\draw[thick, fill=white] (3,0) circle (0.05);
\draw[thick, fill=white] (4,0) circle (0.05);
\draw[thick, fill=white] (0,0) circle (0.05);
\draw[thick, fill=white] (1,1.5) circle (0.05);
\draw[thick, fill=white] (2,1.5) circle (0.05);
\draw[thick, fill=white] (3,1.5) circle (0.05);
\draw[thick, fill=white] (4,1.5) circle (0.05);
\draw[thick, fill=white] (0,1.5) circle (0.05);
\draw[thick, fill=white] (1,3) circle (0.05);
\draw[thick, fill=white] (2,3) circle (0.05);
\draw[thick, fill=white] (3,3) circle (0.05);
\draw[thick, fill=white] (4,3) circle (0.05);
\draw[thick, fill=white] (0,3) circle (0.05);
\draw[thick, fill=white] (1,-1.5) circle (0.05);
\draw[thick, fill=white] (2,-1.5) circle (0.05);
\draw[thick, fill=white] (3,-1.5) circle (0.05);
\draw[thick, fill=white] (4,-1.5) circle (0.05);
\draw[thick, fill=white] (0,-1.5) circle (0.05);
\draw[thick, fill=white] (1,-3) circle (0.05);
\draw[thick, fill=white] (2,-3) circle (0.05);
\draw[thick, fill=white] (3,-3) circle (0.05);
\draw[thick, fill=white] (4,-3) circle (0.05);
\draw[thick, fill=white] (0,-3) circle (0.05);
\draw[above] (0,0) node {\(a_{0,5}\)};
\draw[above] (1,0) node {\(a_{0,4}\)};
\draw[above] (2,0) node {\(a_{0,3}\)};
\draw[above] (3,0) node {\(a_{0,2}\)};
\draw[above right] (3.85,0) node {\(a_{0,1}\)};
\draw[above] (0,1.5) node {\(a_{1,5}\)};
\draw[above] (1,1.5) node {\(a_{1,4}\)};
\draw[above] (2,1.5) node {\(a_{1,3}\)};
\draw[above] (3,1.5) node {\(a_{1,2}\)};
\draw[above right] (3.85,1.5) node {\(a_{1,1}\)};
\draw[above] (0,3) node {\(a_{2,5}\)};
\draw[above] (1,3) node {\(a_{2,4}\)};
\draw[above] (2,3) node {\(a_{2,3}\)};
\draw[above] (3,3) node {\(a_{2,2}\)};
\draw[above] (4,3) node {\(a_{2,1}\)};
\draw[above] (0,-1.5) node {\(a_{-1,5}\)};
\draw[above] (1,-1.5) node {\(a_{-1,4}\)};
\draw[above] (2,-1.5) node {\(a_{-1,3}\)};
\draw[above] (3,-1.5) node {\(a_{-1,2}\)};
\draw[above right] (3.85,-1.5) node {\(a_{-1,1}\)};
\draw[above] (0,-3) node {\(a_{-2,5}\)};
\draw[above] (1,-3) node {\(a_{-2,4}\)};
\draw[above] (2,-3) node {\(a_{-2,3}\)};
\draw[above] (3,-3) node {\(a_{-2,2}\)};
\draw[above right] (3.85,-3) node {\(a_{-2,1}\)};
\draw (0,3.7) node {\(\vdots\)};
\draw (1,3.7) node {\(\vdots\)};
\draw (2,3.7) node {\(\vdots\)};
\draw (3,3.7) node {\(\vdots\)};
\draw (4,3.7) node {\(\vdots\)};
\draw (0,-3.4) node {\(\vdots\)};
\draw (1,-3.4) node {\(\vdots\)};
\draw (2,-3.4) node {\(\vdots\)};
\draw (3,-3.4) node {\(\vdots\)};
\draw (4,-3.4) node {\(\vdots\)};
\draw[ultra thick, -to] (4.5,0) -- (4.85,0);
\draw[ultra thick, -to] (4.5,1.5) -- (4.85,1.5);
\draw[ultra thick, -to] (4.5,3) -- (4.85,3);
\draw[ultra thick, -to] (4.5,-1.5) -- (4.85,-1.5);
\draw[ultra thick, -to] (4.5,-3) -- (4.85,-3);
\draw[ultra thick, -to] (-1,3.735) -- (-1,-3.6);
\end{tikzpicture}
\caption{A graph drawing of \(\inn{V;R_S}\).}
\label{fig}
\end{figure}

This graph drawing uses some conventions from \cite{tot}. To reduce clutter, we exclude loops, as there are loops at every vertex.   The thick vertical arrow indicates that a vertex points at each vertex below it. For example, \(a_{1,1}\) points at \(a_{0,1}\), \(a_{0,2}\), and \(a_{-1,1}\). The thick horizontal arrows indicate that a vertex points at each vertex to its right. For example, \(a_{1,3}\) points at \(a_{1,2}\) and \(a_{1,1}\). Dashed arrows indicate edges whose inclusion depends on the choice of \(S\). For example, \(a_{0,1}\) points at \(a_{1,3}\) and \(a_{1,1}\) points at \(a_{2,3}\) and \(a_{2,3}\) when \(3 \in S\), but \(a_{0,1}\) and \(a_{1,1}\) do not point at \(a_{1,3}\) and \(a_{2,3}\), respectively, when \(3 \notin S\).

For another example, we will list all of the vertices that \(a_{0,1}\) points at. When \(m \in \mathbb{N}\) is even, \(a_{0,1}\) always points at \(a_{1,m}\). If \(m \in \bb{O}\), then \(a_{0,1}\) points at \(a_{1,m}\) when \(m \in S\). If \(p < 0\) and \(m \in \mathbb{N}\), then \(a_{0,1}\) always points at \(a_{p,m}\). Lastly, \(a_{0,1}\) always points at \(a_{0,1}\) and \(a_{0,2}\).

Firstly, we will need to check that these relations are total.

\begin{lemma}\label{reftot}
Let \(S \subseteq \bb{O}\). Then \(R_S\) is total.
\end{lemma}

\begin{proof}
Let \(p,q \in \bb{Z}\) and let \(m,n \in \bb{N}\). As \(m \ge m\), we have \((a_{p,m}, a_{p,m}) \in R_S\), so \(R_S\) is reflexive. Assume that \(a_{p,m} \neq a_{q,n}\). If \(p = q\), we must have \(m \neq n\), hence \(m> n\) or \(m < n\), so \((a_{p,m}, a_{q,n}) \in R_S\) or \((a_{q,n}, a_{p,m}) \in R_S\). Similarly, when \(p \neq q\), we have \(p > q\) or \(p < q\), so  \((a_{p,m}, a_{q,n}) \in R_S\) or \((a_{q,n}, a_{p,m}) \in R_S\). Combining these results, we find that \(R\) is total, which is what we wanted. \end{proof}

Using Proposition \ref{tense} and Proposition \ref{disc}(1), we get the following result.
 
\begin{corollary}\label{ourdisc}
Let \(S \subseteq \bb{O}\). Then every element of \(\IS(\bf{F}_S)\) is simple.
\end{corollary}

The full complex algebras of these frames are too big for our purposes; we will instead work with the subalgebras generated by all of their atoms. Due to the difference in structure between our frames and the frames in \cite{tot}, these algebras will be somewhat more difficult to describe explicitly than the  algebras in \cite{tot}. However, these explicit descriptions are more convenient to work with, so we will use them to define the algebras we will be working with.

\begin{definition}
For all \(p \in \bb{Z}\), \(m \in \bb{N}\), and \(S \subseteq \bb{O}\), let \(V_p \deq \{ a_{p,n} \mid n \in \bb{N}\}\), \(A_{p,m} \deq \{a_{p,m}\}\), \(D_p \deq \{a_{q,n} \mid  q \le p, n \in \bb{N}\}\), \(U_p \deq \{a_{q,n} \mid  q \ge p, n \in \bb{N} \}\), \(S_{p,m} \deq \{ a_{p,n} \mid   n  \in S_\bb{E}, n \ge m \}\), and \(\bar{S}_{p,m} \deq  \{ a_{p,n} \mid  n \notin S_\bb{E}, n>1, n \ge m\}\). Now, for each \(S \subseteq \bb{O}\), let
\[
\s{S}_S \deq \{A_{p,m}, S_{p,m}, \bar{S}_{p,m}  \mid p \in \bb{Z}, m \in \bb{N}\} \cup \{D_p, U_p \mid p \in \bb{Z}\}
\]
and let \(\s{B}_S\) be the set of finite unions of elements of \(\s{S}_S\).
\end{definition}

We aim to show that \(\s{B}_S\) is a subuniverse of \(\bf{F}_S\), for all \(S \subseteq \bb{O}\). Firstly, we will describe the action of the operators of the elements of \(\s{S}_S\). To avoid double subscripts, we write \(f_S\) and \(g_S\) in place of \(f_{R_S}\) and \(g_{R_S}\), respectively.

\begin{lemma}\label{fg}
Let \(S \subseteq \bb{O}\), let \(p \in \bb{Z}\), let \(m \in \bb{N}\), and define \(T \deq \{n \in \bb{N} \setminus S_\bb{E} \mid n \ge m\}\). Then
\begin{enumerate}
\item
\(f_S(A_{p,1}) = A_{p,1} \cup A_{p,2} \cup D_{p-1} \cup S_{p+1,1}\),
\item
\( f_S(A_{p,m}) = \bigcup\{A_{p,n} \mid  n \le m+1\} \cup D_{p-1}\) if  \(m > 1\),
\item
\(f_S(D_p) = D_p \cup S_{p+1,1}\),
\item
\(f_S(U_p) = V\),
\item
\(f_S(S_{p,m}) = D_p\),
\item
\(f_S(\bar{S}_{p,m}) =  \varnothing\)  if \( T = \varnothing\),
\item
\(f_S(\bar{S}_{p,m}) = \bigcup \{ A_{p,n} \mid n \le \max(T)+1\} \cup D_{p-1}\) if  \(T \neq \varnothing\)  is finite,
\item
\(f_S(\bar{S}_{p,m}) =  D_p\) if \(T\) is infinite,
\item
\(g_S(A_{p,1}) = U_p\),
\item
\(g_S(A_{p,2}) = A_{p-1,1} \cup U_p\),
\item
\(g_S(A_{p,m}) = U_{p+1} \cup S_{p,m-1} \cup \bar{S}_{p,m-1}\) if \(m > 1\) and \(m \notin S_\bb{E}\),
\item
\(g_S(A_{p,m}) = A_{p-1,1} \cup U_{p+1} \cup S_{p,m-1} \cup \bar{S}_{p,m-1}\) if \(m>2\) and  \(m \in S_\bb{E}\),
\item
\(g_S(D_p) = V\),
\item
\(g_S(U_p) = A_{p-1,1} \cup U_p\),
\item
\(g_S(S_{p,m}) = A_{p-1,1} \cup U_p\),
\item
\(g_S(\bar{S}_{p,m}) = \varnothing \) if \(T = \varnothing\), 
\item
\(g_S(\bar{S}_{p,m}) = S_{p,\min(T)-1} \cup \bar{S}_{p,\min(T)-1} \cup U_{p+1}\) if  \(T \neq \varnothing\).
\end{enumerate}
\end{lemma}
\begin{proof}
If \(X \subseteq V\), then \(f_S(X)\) is the set of vertices that are pointed at by \(X\), while \(g_S(X)\) is the set of vertices that point at \(X\). From this and Figure~\ref{fig}, the required results follow immediately. \end{proof}

\begin{lemma}\label{desc}
Let \(S \subseteq \bb{O}\). Then \(\s{B}_S\) is the subuniverse of \(\bf{F}_S\) generated by \(\s{S}_S\).
\end{lemma}

\begin{proof} It is clear that \(\s{S}_S \subseteq \s{B}_S\) and that any subuniverse of \(\bf{F}_S\) that extends \(\s{S}_S\) also extends \(\s{B}_S\), so it remains to show that \(\s{B}_S\) is a subuniverse of \(\bf{F}_S\).

By definition, \(\s{B}_S\) is the set of all finite unions of elements of \(\s{S}_S\). Thus, \(\s{B}_S\) is closed under (binary) union and \(\varnothing \in \s{B}_S\). 

By distributivity, to show that \(\s{B}_S\) is closed under (binary) intersection, we only need to show that \(X \cap Y \in \s{B}_S\) if \(X,Y \in \s{S}_S\).  Let \(p \in \bb{Z}\) and let \(m \in \bb{N}\). If \(X \in \s{S}_S\), then \(A_{p,m} \cap X = \varnothing\) or \(A_{p,m} \cap X = A_{p,m}\), hence \(A_{p,m} \cap X \in \s{B}_S\). Clearly, \(U_p \cap U_q = U_{\max(p,m)}\) and \(U_p \cap D_q = \bigcup\{ V_r \mid p \le r \le q\} \), for all \(q \in \bb{Z}\). If \(q \in \bb{Z}\), \(n \in \bb{N}\) and \(X \in \{ S_{q,n}, \bar{S}_{q,n}\}\), then we have \(U_p \cap X = X\) when \(q \ge p\) and \(U_p \cap X = \varnothing\) when \(q < p\). Thus,   \(U_p \cap X \in \s{B}_S\) when \(X \in \s{S}_S\).  If \(q \in \bb{Z}\), then it is clear that \(D_p \cap D_q = D_{\min(p,q)}\).  If \(X \in \{S_{q,n}, \bar{S}_{q,n}\}\), for some \(q \in \bb{Z}\) and \(n \in \bb{N}\), then \(D_p \cap X = X\) if \(q \le p\) and \( D_p \cap X = \varnothing\) if \(q > p\). From this, it follows that  \(D_p \cap X \in \s{B}_S\), for all \(X \in \s{S}_S\). If \(q \in \bb{Z}\) and \(n \in \bb{N}\), then we have  \(S_{p,m} \cap S_{q,n} = S_{p,\max(m,n)} \) when \(q = p\) and \(S_{p,m} \cap S_{q,n} = \varnothing\) when \(q \neq p\). Clearly, \(S_{p,m} \cap \bar{S}_{q,n} = \varnothing\), for all \(p \in \bb{Z}\) and \(n \in \bb{N}\), so \(S_{p,m} \cap X \in \s{B}_S\) if  \(X \in \s{S}_S\).  If \(q \in \bb{Z}\) and \(n \in \bb{N}\), then we have \(\bar{S}_{p,m} \cap \bar{S}_{q,n} = S_{p,\max(m,n)}\) when \(q=p\) and \(\bar{S}_{p,m} \cap \bar{S}_{q,n} = \varnothing\) otherwise, so \(\bar{S}_{p,m} \cap X \in \s{B}_S\), for all \(X \in \s{S}_S\). Combining these results, we find that \(\s{B}_S\) is closed under intersection.

From De Morgan's laws and the observations above, to show the closure of \(\s{B}_S\) under forming complements, it will be enough to show that \(X^c \in \s{B}_S\) if \(X \in \s{S}_S\). As \(A_{p,m}^c = \bigcup \{A_{p,n} \mid  n < m \} \cup U_{p+1} \cup D_{p-1} \cup S_{p,m+1} \cup \bar{S}_{p,m+1}\), \(U_p^c = D_{p-1}\),  \(D_p^c = U_{p+1}\), \(S_{p,m}^c = \bigcup \{A_{p,n} \mid n < m\} \cup \bar{S}_{p,m} \cup U_{p+1} \cup D_{p-1}\) and \(\bar{S}_{p,m}^c = \bigcup \{A_{p,n} \mid  n < m\} \cup S_{p,m} \cup U_{p+1} \cup D_{p-1}\), it follows that \(\s{B}_S\) is closed under forming complements. As \(\varnothing \in \s{B}_S\), this result  tells us that \(V \in \s{B}_S\).

By additivity, to show that \(\s{B}_S\) is closed under \(f_S\) and \(g_S\), we only need to check that \(f_S(X), g_S(X) \in \s{B}_S\) if \(X \in \s{S}_S\). This follows from Lemma \ref{fg}, so \(\s{B}_S\) is indeed a subuniverse of \(\bf{F}_S\), which is what we wanted.  \end{proof}

This result allows us to make the following definition.

\begin{notation}
Let \(S \subseteq \bb{O}\). The subalgebra of \(\bf{F}_S\) with universe \(\s{B}_S\) will be denoted by \(\bf{B}_S\).
\end{notation}

Before we show that these algebras have the properties we want, we will show that they are generated by any element of the form \(V_p\), for some \(p \in \bb{Z}\).

\begin{lemma}\label{4or5} 
Let \(S \subseteq \bb{O}\), let \(X \subseteq V\) and assume that there is a maximal \(p \in \bb{Z}\) with  \(V_p \cap X \neq \varnothing \), say \(q\). Then
\begin{enumerate}
\item
\(f_S^4(X) \cap f_S^2(X)^c = V_{q+2}\) if \(a_{q,1} \in X\),
\item
\(f_S^5(X) \cap f_S^3(X)^c = V_{q+2}\) if \(a_{q,1} \notin X\).
\end{enumerate}
\end{lemma}

\begin{proof}
Firstly, assume that \(a_{q,1} \in X\). Based on Figure \ref{fig},   \(f_S^2(X) = D_{q+1}\) and \(f_S^4(X) = D_{q+2}\), which implies that \(f_S^4(X) \cap f_S^2(X)^c = V_{q+2} \). Thus, (1) holds.

Now, assume that we have \(a_{q,1} \notin X\). Similarly to the previous case,  \(f_S^3(X) = D_{q+1}\) and \(f_S^5(X) = D_{q+2}\), hence \(f_S^5(X) \cap f_S^3(X)^c = V_{q+2} \). Thus, (2) also holds. \end{proof}

We will need the following argument later, so we isolate it here. Firstly, we define some terms (in the signature \(\{\j,\m,{'},f,g,0,1\}\) of tense algebras).

\begin{definition}
\begin{enumerate}
\item
Let \(\beta(x)\deq f^4(x) \m f^2(x)'\).
\item
Let \(\sigma(x)\deq f(x) \m (x \j g^2(\beta(x)) \j f^4(g^{10}(\beta(x)) \m g^{8}(\beta(x))'))'\).
\item
Let \(\nu_3(x)\deq f(\sigma(x)) \m f(x)'\).
\item
Let \(\nu_4(x) \deq f(\nu_3(x)) \m f(\sigma(x))'\).
\item
For each \(n \ge 5\), let \(\nu_n \deq f(\nu_{n-1}(x)) \m f(\nu_{n-2}(x))'\).
\end{enumerate}
\end{definition}

\begin{lemma}\label{steps}
Let \(S \subseteq \bb{O}\) and let \(p \in \bb{Z}\). Then we have \(\sigma(A_{p,1}) = A_{p,2}\) and \(\nu_n(\sigma(x)) = A_{p,n}\), for all \(n \ge 3\).
\end{lemma}

\begin{proof}
By Lemma \ref{4or5}(1), we have \(\beta(A_{p,1}) = V_{p+2}\). So, based on Figure \ref{fig},  \(g_S^2(\beta(A_{p,1})) = U_{p+1}\). Similarly, \(g_S^8(V_{p+2}) = U_{p-2}\) and \(g_S^{10}(V_{p+2}) = U_{p-3}\), hence \(f_S^4(g_S^{10}(\beta(A_{p,1}) \cap g_S^8(\beta(A_{p,1}))^c) = f_S^4(V_{p-3}) = D_{p-1}\). By Lemma \ref{fg}(1), we have \(\sigma(A_{p,1}) = A_{p,2}\), as claimed.

Next, we will use a (strong) inductive argument for the second claim.  By Lemma \ref{fg}(1) and Lemma \ref{fg}(2),  \(\nu_3(A_{p,1}) = f_S(A_{p,2}) \cap f_S(A_{p,1})^c = A_{p,3}\) and \(\nu_4(A_{p,1}) = f_S(A_{p,3}) \cap f_S(A_{p,2})^c = A_{p,3}\) as \(\sigma(A_{p,1}) = A_{p,2}\). Let \(n \ge 5\) and assume that \(\nu_m(A_{p,1}) = A_{p,m}\), for all \(4 \le m \le n\). From this assumption and Lemma \ref{fg}(2), it follows that \(\nu_{n+1}(A_{p,1}) = f_S(A_{p,n}) \cap f_S(A_{p,n-1})^c = A_{p,n+1}\). Thus, \(\nu_m(A_{p,1}) = A_{p,m}\), for all \(m \ge 3\), as claimed. \end{proof}

\begin{lemma}\label{bgen}
Let \(S \subseteq \bb{O}\) and let \(p \in \bb{Z}\). Then \(\s{B}_S\) is the subuniverse of \(\bf{B}_S\) generated by \(V_p\).
\end{lemma}

\begin{proof}
Let \(\s{V}_p\) denote the subuniverse of \(\bf{B}\) generated by \(V_p\). By Lemma \ref{desc},  it will be enough to show that \(\s{S}_S \subseteq \s{V}_p\).

Firstly, we claim that \(V_q \in \s{V}_p\), for every \(q \in \bb{Z}\). Based on Figure \ref{fig}, Lemma \ref{fg}(3) and Lemma \ref{fg}(5), we have \(f_S^{2m}(V_p) = D_{p+m}\), for each \(m \in \bb{N}\). This implies that \(V_{p+m+1} = f_S^{2m+2}(V_p) \cap f_S^{2m}(V_p)^c  \in \s{V}_p\), for every \(m \in \bb{N}\), hence we have \(V_{q} \in \s{V}_p\), for every \(q \ge m+2\). Similarly, if \(q \in \bb{Z}\) and \(m \in \bb{N}\), then \(g_S^{2m}(V_q) = U_{q-m}\), hence \(V_{q-m-1} = g_S^{2m+2}(V_q)^c \cap g_S^{2m}(V_q) \in \s{V}_p\). Thus, we must have  \(V_q \in \s{V}_p\), for all \(q\in \bb{Z}\), as claimed.

Based on Figure \ref{fig}, Lemma \ref{fg}(3) and Lemma \ref{fg}(5),   \(f_S^2(V_{q-1}) = D_q\), for each \(q \in \bb{Z}\). Similarly, we also have  \(g_S^2(V_{q+1}) = U_q\), for every \(q \in \bb{Z}\). Hence, by the previous result, we have \(D_q, U_q \in \s{V}_p\), for all \(q \in \bb{Z}\).

By Lemma \ref{fg}(13) and the previous result, \(A_{p,1} = g(U_{q+1}) \cap U_{q+1}^c  \in \s{V}_p\), for all \(q \in \bb{Z}\). So, by Lemma \ref{steps}, we have \(A_{p,n} \in \s{V}_p\), for all \(p \in \bb{Z}\) and \(n \in \bb{N}\).

Based on Lemma \ref{fg}(3),  \(S_{q,m} = f_S(D_{q-1}) \cap (\bigcup \{ A_{q,n} \mid n < m\} \cup D_{q-1})^c\). Hence, by the above results, we must have  \(S_{q,m} \in \s{V}_p\), for all \(q\in \bb{Z}\) and \(m \in \bb{N}\).

Clearly, we must have \(\bar{S}_{q,m} = (\bigcup \{ A_{q,n} \mid n < m\} \cup D_{q-1} \cup U_{q+1} \cup S_{q,m} )^c \), for all \(q \in \bb{Z}\) and \(m \in \bb{N}\). So, based the above results, we must have \(\bar{S}_{q,m} \in \s{V}_p\), for all \(q \in \bb{Z}\) and \(m \in \bb{N}\). 

Based on these results, \(\s{V}_p = \s{B}_S\), which is what we wanted to show. \end{proof}

Now we will shift our focus to varieties. To make use of Proposition~\ref{disc}(3), we will need a number of intermediate results.

\begin{lemma}\label{XorXc}
Let \(S \subseteq \bb{O}\) and let \(X \in \s{B}_S\). Then \(f_S(X) \neq V\) or \(f_S(X^c) \neq V\).
\end{lemma}

\begin{proof}
By Lemma \ref{desc}, \(X\) and \(X^c\) can be represented as unions of finite subsets of \(\s{S}_S\). Clearly, only one of the unions will involve an element of \(\{U_p \mid p \in \bb{Z}\}\). So, based on Figure \ref{fig}, we have \(f_S(X) \neq V\) or \(f_S(X^c) \neq V\), as required.\end{proof}

\begin{lemma}\label{max}
Let \(S \subseteq \bb{O}\) and let \(X \in \s{B}_S\) such that \(X \neq \varnothing\) and \(f_S(X) \neq V\). Then there is a maximal \(p \in \bb{Z}\) with \(V_p \cap X \neq \varnothing\).
\end{lemma}

\begin{proof}
By Lemma \ref{desc}, \(X\) can be represented as the union of a finite subset of \(\s{S}\). By assumption, \(f_S(X) \neq V\), so Lemma \ref{fg}(iii) tells us that such a representation cannot involve an element of \(\{U_p \mid p \in \bb{Z}\}\). Since \(X \neq \varnothing\), there is maximal \(p \in \bb{Z}\) with \(V_p \cap X \neq \varnothing\), as claimed.  \end{proof}

Using Lemma \ref{4or5} and the preceding pair of results, it is easy to verify our previous claim that \(\bf{B}_S\) is the subalgebra of \(\bf{F}_S\) generated by its atoms, or by any element of \(\s{B}_S\), for each \(S \subseteq \bb{O}\). The following result will allow us to obtain similar results for ultrapowers.

\begin{lemma}\label{auto}
Let \(S \subseteq \bb{O}\) and let \(p,q \in \bb{Z}\). 
\begin{enumerate}
\item
We have \(\s{B}_S = \{ t^{\bf{B}_S}(V_p) \mid t \trm{ is a unary term}\}\).
\item
If \(t\) and \(s\) are unary terms with  \(t^{\bf{B}_S}(V_p) = s^{\bf{B}_S}(V_p)\), then we have \(t^{\bf{B}_S}(V_q) = s^{\bf{B}_S}(V_q)\).
\item
There is an automorphism of \(\bf{B}_S\) that maps \(V_p\) to \(V_q\).
\end{enumerate}
\end{lemma}

\begin{proof}
The first statement is an immediate consequence of Lemma \ref{desc}, while (2) is evident from the self similarity of \(\inn{V;R_S}\).

From (1) and (2), it follows that we can define a map \(\mu \colon \s{B}_S \to \s{B}_S\) by setting \(\mu(t^\bf{B}(V_p)) = t^{\bf{B}}(V_q)\), for every unary term \(t\). Combining (1) and (2), we find that \(\mu\) is a bijection. Based on (2), \(\mu\) is an endomorphism of \(\bf{B}_S\). Thus, \(\mu\) is an automorphism of \(\bf{B}_S\), so (3) holds.
\end{proof}

\begin{lemma}\label{embed}
Let \(S \subseteq \bb{O}\), let \(I\) be a non-empty set, let \(\s{U}\) be an ultrafilter over \(I\), and let \(X\in \s{B}_S^I\) with \(X/\s{U} \neq 0\) and \(X/\s{U} \neq 1\). Then \(\bf{B}_S\) embeds into the subalgebra of \(\bf{B}_S^I/\s{U}\) generated by \(X/\s{U}\). 
\end{lemma}

\begin{proof}
By Lemma \ref{XorXc} and {\L}o\'s's Theorem,  \(f(X/\s{U}) \neq 1\) or \(f(X/\s{U}')\neq 1\). Without loss of generality, we can assume that \(f(X/\s{U}) \neq 1\). By Lemma~\ref{4or5} and Lemma~\ref{max}, either \(\{i \in I \mid f_S^4(X(i)) \cap f_S^2(X(i))^c =V_p, \trm{ for some }p \in \bb{Z}\}\) or \(\{i \in I \mid f_S^5(X(i)) \cap f_S^3(X(i))^c =V_p, \trm{ for some }p \in \bb{Z}\}\) is an element of \(\s{U}\). Based on Lemma~\ref{auto}(3),  \(\bf{B}_S\) embeds into the subalgebra of \(\bf{B}_S^I/\s{U}\) generated by \(X/\s{U}\), as claimed. 
\end{proof}

Note that we only needed the fact that ultrafilters are prime filters, hence this result applies to principal ultrafilters.

\begin{lemma}\label{cov}
Let \(S \subseteq \bb{O}\). Then \(\Var(\bf{B}_S)\) covers \(\Var(\bf{T}_0)\) in \(\boldsymbol{\Lambda}_\textup{TTA}\). Further, \(\Var(\bf{B}_S)\) is join-irreducible.
\end{lemma}

\begin{proof}
It is clear that \(\{\varnothing, V\}\) is a subuniverse of  \(\bf{B}_S\), and that the corresponding subalgebra of \(\bf{B}_S\) is isomorphic to \(\bf{T}_0\). By J\'onsson's Theorem,  \(\Si(\Var(\bf{T}_0)) = \I(\bf{T}_0)\), so by Lemma \ref{ourdisc}, we must have \(\Var(\bf{T}_0) \subsetneq \Var(\bf{B}_S)\). Let \(\c{V} \in \Lambda_{\trm{TTA}}\) with \(\Var(\bf{T}_0) \subsetneq \c{V} \subseteq \Var(\bf{B}_S)\). Clearly, \(\Si(\c{V}) \subseteq \Si(\Var(\bf{B}_S))\). By Proposition \ref{disc}(3),  \(\Si(\Var(\bf{B}_S)) =
\ISU(\bf{B}_S)\), hence \(\Si(\c{V}) \subseteq \ISU(\bf{B}_S)\). Since \(\Var(\bf{T}_0) \subsetneq \c{V}\), there is some \(\bf{A} \in \Si(\c{V})\) with more than 2 elements. By Lemma \ref{embed}, \(\bf{B}_S\) embeds into \(\bf{A}\), so \(\bf{B}_S \in \c{V}\), hence \(\c{V} = \Var(\bf{B}_S)\). Thus, \(\Var(\bf{B}_S)\) is a cover of \(\Var(\bf{T}_0)\) in \(\boldsymbol{\Lambda}_\trm{TTA}\), as required. Based on these arguments, it is evident that  \(\Var(\bf{B}_S)\) is also join-irreducible. \end{proof} 

Lastly, we will need to show that these varieties are distinct. The following pair of results effectively reduce this problem to showing that \(\bf{A}_S\) and \(\bf{A}_T\) are not elementarily equivalent, for all distinct \(S,T \subseteq \bb{O}\).

\begin{lemma}\label{homisiso}
Let \(S,T \subseteq \bb{O}\) and let \(\mu \colon \bf{B}_S \to \bf{B}_T\) be a homomorphism. Then \(\mu\) is an isomorphism.
\end{lemma}

\begin{proof}
Since \(\bf{B}_S\) and \(\bf{B}_T\) are non-trivial, the kernel of \(\mu\) must be non-zero. From Corollary \ref{ourdisc}, \(\bf{B}_S\) is simple, so the kernel of \(\mu\) is the identity relation. This implies that \(\mu\) is an embedding, so \(\varnothing \subsetneq \mu(V_0) \subsetneq V\). By Lemma \ref{4or5}, Lemma \ref{bgen}, Lemma \ref{XorXc} and Lemma \ref{max}, we must have \(\mu[\s{B}_S] = \s{B}_T\). Thus,  \(\mu\) is surjective, hence \(\mu\) is an isomorphism, as required. \end{proof}

Using Lemma \ref{steps}, we will construct some useful first-order formulae (again, in the signature \(\{\j,\m,{'},f,g,0,1\}\) of tense algebras). To avoid confusion, we will use \(\curlyvee\) for logical disjunction and  \(\curlywedge\) for logical conjunction.

\begin{definition}
\begin{enumerate}
\item
Let \(\alpha(x) \deq x \not\approx 0 \curlywedge (\forall y \colon x \m y \approx 0 \curlyvee x \m y \approx x)\).
\item
Let \(\varphi(x)\deq \alpha(x) \curlywedge \neg (\exists w, y, z \colon \alpha(w) \curlywedge \alpha(y) \curlywedge \alpha(z) \curlywedge f(x) \m g(x) \approx w \j y \j z)\). 
\item
For each \(n \ge 3\), let \(\tau_n(x)\deq \varphi(x) \curlywedge \nu_n(x) \m f(g^2(x)\m g(x)') \not\approx 0\).
\end{enumerate}
\end{definition}

\begin{lemma}\label{sent}
Let \(S \subseteq \bb{O}\), let \(n \ge 3\) and let \(X \in \s{B}_S\). Then
\begin{enumerate}
\item
\(\bf{B}_S \models \varphi[X]\) if and only if \(X = A_{p,1}\), for some \(p \in \bb{Z}\),
\item
\(\bf{B}_S \models \tau_n[X]\) if and only if \(n \in S_\bb{E}\) and \(X = A_{p,1}\), for some \(p \in \bb{Z}\).

\end{enumerate}
\end{lemma}

\begin{proof}
If \(\bf{T}\) is a tense algebra and \(x \in T\), then \(\bf{T} \models \alpha[x]\) if and only if \(x\) is an atom, hence \(\bf{B} \models \alpha[X]\) if and only if \(X = A_{p,n}\), for some \(p \in \bb{Z}\) and \(n \in \bb{N}\). By Lemma \ref{fg}(1) and Lemma \ref{fg}(9),  \(f_S(A_{p,n}) \cap g_S(A_{p,n}) = A_{p,1} \cup A_{p,2} \cup S_{p,1}\) if \(n=1\) and \(p \in \bb{Z}\). Similarly,  \(f_S(A_{p,n}) \cap g_S(A_{p,n}) = A_{p,n-1} \cup A_{p,n} \cup A_{p,n+1}\) if \(n > 1\)~and~\(p \in \bb{Z}\). Combining these results, we see that (1) holds.

Based on Lemma \ref{fg}(9) and Lemma \ref{fg}(13), we have \(g_S(A_{p,1}) = U_p\) and \(g_S^2(A_{p,1}) = g_S(U_p) = A_{p-1,1} \cup U_p\), for each \(p \in \bb{Z}\). So, by Lemma \ref{fg}(1),  \(f_S(g_S^2(A_{p,1}) \cap g_S(A_{p,1})^c) = f_S(A_{p-1,1}) =  A_{p-1,1} \cup A_{p-1,2} \cup D_{p-2} \cup S_{p,1}\) when \(p \in \bb{Z}\). By Lemma \ref{steps},  \(\tau_n(A_{p,1}) \cap f_S(g_S^2(A_{p,1}) \cap g_S(A_{p,1})^c)  \neq \varnothing\) if and only if \(n \in S_\bb{E}\), as \(\tau_n(A_{p,1}) = A_{p,n}\). Based on this and (1), (2) also holds. \end{proof}

Now we have the results we need to show that our varieties are distinct.

\begin{lemma}\label{diff}
Let \(S,T \subseteq \bb{O}\) with \(S \neq T\). Then \(\Var(\bf{B}_S) \neq \Var(\bf{B}_T)\).
\end{lemma}

\begin{proof}
Without loss of generality, we can assume that \(S \nsubseteq T\), since \(S \neq T\). Let \(n \in \bb{N}\) with \(n \in S\) and \(n \notin T\). By Lemma \ref{sent},  \(\bf{B}_S \models \exists x \colon \tau_n(x)\) and \(\bf{B}_T \nmodels \exists x \colon \tau_n(x)\), hence \(\bf{B}_S\) and \(\bf{B}_T\) are not elementarily equivalent. Thus, \(\bf{B}_S\) and \(\bf{B}_T\) are not isomorphic, so by Lemma \ref{homisiso}, \(\bf{B}_S\) does not embed into \(\bf{B}_T\). Based on Proposition \ref{disc}(3), Corollary \ref{ourdisc} and Lemma \ref{embed},  \(\bf{B}_T \notin \Var(\bf{B}_S)\), so \(\Var(\bf{B}_S) \neq \Var(\bf{B}_T)\), as claimed.  \end{proof}

Now we just need to put on the finishing touches.

\begin{theorem}\label{mainTA}
\(\Var(\bf{T}_0)\) has  \(2^{\aleph_0}\) join-irreducible covers in \(\boldsymbol{\Lambda}_\textup{TTA}\) and \(\boldsymbol{\Lambda}_\textup{TA}\).
\end{theorem}

\begin{proof}
Let \(C\) denote the set of join-irreducible covers of \(\Var(\bf{T}_0)\) in \(\boldsymbol{\Lambda}_\textup{TTA}\). Combining Lemma \ref{cov} and Lemma \ref{diff}, we find that \(|C| \ge 2^{\aleph_0}\). It is easy to see that there are at most \(2^{\aleph_0}\) sets of equations in a countable signature (up to replacing variables), hence \(|C| \le |\Lambda_\textup{TTA}| \le 2^{\aleph_0}\). Therefore \(|C| = 2^{\aleph_0}\), which is what we wanted to show.
\end{proof}

Combining Proposition \ref{covprop} and Theorem \ref{mainTA}, we obtain our main result.

\begin{theorem}\label{mainSA}
\(\Var(\bf{A}_3)\) has exactly \(2^{\aleph_0}\) join-irreducible covers in \(\boldsymbol{\Lambda}_\textup{RSA}\), \(\boldsymbol{\Lambda}_\textup{SA}\), and  \(\boldsymbol{\Lambda}_\textup{NA}\).
\end{theorem}

We can also use Theorem \ref{mainTA} to obtain similar results on the subvariety lattices of other varieties of \(r\)-algebras and nonassociative relation algebras; see Section 4 of \cite{tot} for more details.

\section{Concluding remarks}

In the previous section we saw that \(\Var(\bf{T}_0)\) has \(2^{\aleph_0}\) covers in \(\boldsymbol{\Lambda}_\textup{TA}\) and  \(\boldsymbol{\Lambda}_\textup{TTA}\), and that \(\Var(\bf{A}_3)\) has \(2^{\aleph_0}\) covers in \(\boldsymbol{\Lambda}_\textup{SA}\) and  \(\boldsymbol{\Lambda}_\textup{RSA}\). However, the problem of completely characterizing these covers remains open. Based on a computer search, it seems that there are finite semiassociative relation algebras that generate covers of \(\Var(\bf{A}_3)\) in \(\boldsymbol{\Lambda}_\textup{SA}\) that are not in the list in \cite{jipphd}. In fact, based on this search, it seems reasonable to conjecture that there are finite algebras with arbitrarily large atom sets that generate covers of \(\Var(\bf{A}_3)\) in  \(\boldsymbol{\Lambda}_\textup{SA}\), so the problem of completely characterizing covers of \(\Var(\bf{A}_3)\) in \(\boldsymbol{\Lambda}_\textup{SA}\) could prove to be quite difficult.

Using Lemma \ref{fg} and the term equivalence in \cite{tot}, it is not too difficult to show that the semiassociative relation algebras corresponding to the tense algebras constructed above are not relation algebras. (For example, \(A_{0,1}(A_{0,1} A_{1,1})\) and \((A_{0,1} A_{0,1}) A_{1,1}\) are always distinct, so associativity fails.) Thus, the results obtained above do not appear to provide any information about the covers of \(\Var(\bf{B}_3)\) in the lattice \(\boldsymbol{\Lambda}_\textup{RA}\) of subvarieties of the variety of relation algebras.

Based on these observations, the following problems seem like reasonable starting points for further research in this area.

\begin{problem}
Classify the covers of \(\Var(\bf{T}_0)\) in \(\boldsymbol{\Lambda}_\textup{TA}\) (or \(\boldsymbol{\Lambda}_\textup{TTA}\)).
\end{problem}

\begin{problem}
Classify the covers of \(\Var(\bf{A}_3)\) in \(\boldsymbol{\Lambda}_\textup{SA}\) (or \(\boldsymbol{\Lambda}_\textup{RSA}\)).
\end{problem}

\begin{problem}
Determine whether or not \(\Var(\bf{A}_3)\) has infinitely many finitely generated covers in \(\boldsymbol{\Lambda}_\textup{SA}\).
\end{problem}

\begin{problem}
Determine the number of covers of \(\Var(\bf{A}_3)\) in \(\boldsymbol{\Lambda}_\textup{RA}\).
\end{problem}

\begin{problem}
Classify the covers of \(\Var(\bf{A}_3)\) in \(\boldsymbol{\Lambda}_\textup{RA}\).
\end{problem}

\subsection*{Acknowledgment}
This is a pre-print of an article published in Algebra Universalis. The final authenticated version is available online at:  \url{https://doi.org/10.1007/s00012-020-0646-9}.

The first author would like to thank Peter Jipsen for recommending \cite{tot}, and Eli Hazel for solving a mysterious {\LaTeX} issue.

\end{document}